\documentclass[11pt]{amsart}
\usepackage{amssymb,amsfonts,amsmath, amscd, mathrsfs, latexsym, amsthm}
\usepackage[dvips]{graphicx}

\usepackage[all]{xy}
\title{Cylindrical homomorphisms and Lawson homology}                 
\theoremstyle{plain} 
\newtheorem{theorem}{Theorem}[section]
\newtheorem{lemma}[theorem]{Lemma} 
\newtheorem{definition}[theorem]{Definition}

\newtheorem{conjecture}[theorem]{Conjecture}
\newtheorem{proposition}[theorem]{Proposition}
\newtheorem{corollary}[theorem]{Corollary}
\newtheorem{remark}[theorem]{Remark}

\newcommand{\tensor}{\otimes}
\author{Mircea Voineagu}

\address {Department of Mathematics, University of Southern California,
Los Angeles, CA  90089}
\email{voineagu@usc.edu}



\begin{document}



\begin{abstract}
We use the cylindrical homomorphism and a geometric construction introduced by J. Lewis to study the Lawson homology groups of certain hypersurfaces $X\subset \mathbb{P}^{n+1}$ of degree $d\leq n+1$. As an application, we compute the rational semi-topological K-theory of a generic cubic of dimension 5, 6 and 8 and, using the Bloch-Kato conjecture, we prove Suslin's conjecture for these varieties. Using the generic cubic sevenfolds, we show that  there are smooth projective varieties  with the lowest non-trivial step in their s-filtration infinitely generated and undetected by the Abel-Jacobi map.  

\end{abstract}
\maketitle
\tableofcontents
 Lawson homology of a projective variety $X$ is defined as the homotopy groups of algebraic cycle spaces of a fixed dimension on $X$. We write $$L _rH _n(X):=\pi _{n-2r} (Z _r(X))$$ and, intuitively, think of an element in this group as a ``family of r-cycles parametrized by a (n-2r)-sphere" \cite{FM2}. These groups, replete with information about the projective variety $X$, are a combination of algebraic and topological information. For example, the algebraic equivalence class of an algebraic  r-cycle can be expressed as a connected component of the topological space of algebraic r-cycles $Z _r(X)$ \cite{F}. In the same flavor, Dold-Thom theorem shows that ``families of 0-cycles parametrized by a n-sphere" are the same as the topological classes in the singular homology of $X^{an}$. The s-map is a map that ``measures" how close Lawson homology groups are to algebraic geometry or topology. We have the following sequence of maps
 $$A _r(X)=L _rH _{2r}(X)\stackrel{s}{\rightarrow} L _{r-1}H _{2r}(X)\stackrel{s}{\rightarrow}..\stackrel{s}{\rightarrow} L _1H _{2r}(X)\stackrel{s}{\rightarrow} H _{2r}(X^{an}).$$
 The composition of the above s-maps gives the usual cycle map between the Chow group of algebraic cycles modulo algebraic equivalence, denoted $A _r(X)$, and the singular homology of the complex points of $X$, written $H _{2r}(X^{an})$ \cite{FM2}. Encoded in the construction of the s-map is the celebrated suspension theorem for algebraic cycles proved by Lawson \cite{L}, the starting point of  Lawson homology. 
 
 Lawson homology with finite coefficients of a smooth projective variety $X$ is proved to be isomorphic, via a Poincare type duality \cite{FL1}, with the motivic cohomology with finite coefficients of $X$ \cite{FW2}. Through this isomorphism, we can study the torsion of the Lawson homology groups by motivic cohomology tools (see Proposition \ref{propo}). In particular, the Beilinson-Lichtenbaum conjecture may be used to identify, for certain indices, Lawson homology groups with finite coefficients with singular homology groups with finite coefficients. As shown in \cite{SV1}, the Beilinson-Lichtenbaum conjecture is equivalent to the Bloch-Kato conjecture, which has been proven by V. Voevodsky and M. Rost (see \cite{Voevo},  \cite{Weib}, \cite{Rost2}). 
 
  The torsion free subgroup of a Lawson homology group may be studied using methods from Hodge theory (see \cite{FHW}, \cite{Voin}, \cite{Hu})). Some of these methods, mostly those designed to study the niveau filtration of singular cohomology and the Generalized Hodge conjecture, may be adapted to work on Lawson homology groups. The generalized cycle maps between the Lawson homology groups of $X$ and the singular homology of $X^{an}$ factor through steps in the niveau filtration of the latter and it is conjectured that these steps give the entire image of these generalized cycle maps \cite{FM}. With finite coefficients, this conjecture follows from the Bloch-Kato conjecture. With integer coefficients, this conjecture follows from the very far reaching Suslin's conjecture (see Conjecture (\ref{sc})). Although Suslin's conjecture and the Generalized Hodge conjecture are not known to be related in some direct way (i.e. neither is known to imply the other), all the varieties proved to satisfy Suslin's conjecture (see \cite{FHW}, \cite{Voin}) were previously known to satisfy the Generalized Hodge conjecture. On the other hand, there are varieties that are known to fulfill the Generalized Hodge conjecture and for which Suslin's conjecture is still not known. For example, abelian varieties as in \cite{Abd} or generic cubics sevenfolds or elevenfolds as in Remark \ref{eiel}. 
  
  The main difficulty that appears in these particular examples lives in understanding the kernels of s-maps. The kernels of  s-maps, in particular the kernels of the generalized cycle maps, are mostly mysterious and expected to be in general very huge (i.e. infinitely generated rational vector spaces). However, Suslin's conjecture predicts that, in the range of indices as in the Beilinson-Lichtenbaum conjecture, these kernels are zero for any projective smooth variety.
  
  In some cases (for example in those of certain projective varieties with ``few" algebraic cycles like rationally connected varieties) these kernels can be understood, at least in some range of indices \cite{Voin}, \cite{FHW}. In this paper, we use a geometrical construction, introduced by J. Lewis \cite{Le1}, to study the Lawson homology groups of a generic hypersurface that has ``enough" $k$-planes, $k\geq 1$. Except in the middle dimension, the Lawson homology of a hypersurface always surjects, at least rationally, onto its singular homology, as an application of the simple structure of its singular homology (see Proposition \ref{p1}). The study of the kernels of s-maps of a smooth projective hypersurface  is highly non-trivial, even for indices outside of the middle dimension (see Corollary \ref{c46}). Applying the method of \cite{Le} (also \cite{BM}, \cite{Le1}) on Lawson homology groups allows us not only to obtain the results of \cite{Le}, but also computations of certain Lawson homology groups. As a consequence, we are also able to compute the semi-topological K-theory for certain generic cubic hypersurfaces. Semi-topological K-theory of a projective smooth variety $X$, denoted $K^{sst} _n(X)$, is a K-theory that lives between the algebraic K-theory of X and the topological K-theory of $X^{an}$ , i.e. 
  $$K^{alg} _n(X)\rightarrow K^{sst} _n(X)\rightarrow ku^{-n}(X^{an})$$
sharing important properties with these two K-theories \cite{FW2}.  
  The main tool to compute it is an ``Atiyah-Hirzebruch" type of spectral sequence from the Lawson homology groups of the smooth projective variety $X$ \cite{FHW}.           
  
  The motivation of this paper was the study of Lawson homology of rationally connected varieties of dimension greater than four. In \cite{Voin}, we computed the Lawson homology of rationally connected varieties of dimension three and the rational Lawson homology of dimension four rationally connected varieties; in particular, we were able to check Suslin's conjecture and Friedlander-Mazur's conjecture in these cases. Some questions appeared: Could we check Suslin's conjecture for rationally connected varieties of higher dimension? Are the s-maps always monomorphisms for a rationally connnected variety of any dimension? The results of  \cite{Voin} anticipate that the Lawson homology of rationally connected hypersurfaces may provide answers to these questions, in particular the Lawson homology of a generic cubic sevenfold and eightfold. In \cite{Voin}, the cases of a cubic fivefold and sixfold were discussed by means of the decomposition of diagonal method and of a result of Esnault, Viehweg and Levine \cite{EVL}. Using the natural and transparent geometric method of J. Lewis \cite{Le}, we obtain more general results (the case of a cubic eightfold, see Corollary \ref{g8}). 

  The paper is divided in four sections and an Appendix. In the first section, we introduce our notations and recall some of the results needed later in the paper. 
  
  In the second section, we recall the geometric construction from J. Lewis \cite{Le} and use it in the context of Lawson homology groups. We will also introduce here the cylindrical homomorphisms, which are the main objects of study in this paper. The results proved in this section give us the main tools that will be used in the fourth section. Part of the proofs needed in this section is given in the Appendix. 
  
  In the third section, we extend a ``weak Lefschetz" type of theorem, proved for Chow groups in \cite{Le}, to Lawson homology groups and discuss some applications of the weak Lefschetz theorem on homology to the s-maps on the Lawson homology of a hypersurface of any dimension and degree.  
  
  The fourth section is devoted to applications of the tools developed in the second and third sections. We compute, in a certain range of indices, the Lawson homology of some rationally connected generic hypersurfaces. As an application we compute the rational Lawson homology and semi-topological K-theory of a generic cubic eightfold (Corollary \ref{g8}), obtaining, in particular, the validity of Suslin's conjecture in this particular case. In the end of this section, based on the results obtained in the case of a generic cubic sevenfold, we remark that there are examples of varieties with the lowest nontrivial step in the s-filtration of a Griffiths group (see (\ref{s-filtration})) an infinitely generated $\mathbb{Q}$-vector space and undetected by the Abel-Jacobi map (see Corollary \ref{aju}).

  This paper uses in an essential way  J. Lewis's geometric construction from \cite{Le1} and \cite{Le} in the context of Lawson homology. We thank J. Lewis for making this interesting construction available. 
  
  We thank Eric Friedlander for constant encouragement and for reading various versions of this paper and making useful comments. We thank Pedro dos Santos for a detailed reading of a previous version of the paper and for his suggestions who clearly improved our paper. We thank Jeremiah Heller and Jian He for many discussions related to the results of this paper. 
        
\section{Notations and Recollection}
  In this section, we will introduce the notations used in the paper  and briefly state some of the results needed later on.

All algebraic varieties in this paper are smooth and irreducible over complex numbers . By $L _pH _q(X)$, $CH _r(X)$, $A _r(X)$, $K^{sst}(X)$, $H^{BM} _q$ and $H _q(X^{an})$ we define Lawson homology, Chow group of algebraic cycles modulo rational equivalence, Chow group of algebraic cycles modulo algebraic equivalence, semi-topological K-theory, Borel-Moore cohomology  and singular homology with integer coefficients.  By $L _pH _q(X) _\mathbb{Q}$ we mean Lawson homology with rational coefficients (similar notation for rational semi-topological K-theory  and rational singular homology). By an isomorphism written like $\stackrel{\mathbb{Q}}{\simeq}$ we mean an isomorphism of rational vector spaces. By a monomorphism written $\stackrel{\mathbb{Q}}{\hookrightarrow}$ we mean a rational monomorphism. We will write $\alpha.\beta$ for the intersection product of two algebraic cycles $\alpha$, $\beta$ in $X$. We will write 1 for the identity map, when there is no confusion. We write $[x]$ for the integer part of the real number $x$. 

We call a hypersurface of dimension n generic if it belongs to a point in a non-empty Zariski open subset of the variety of hypersurfaces of degree $d$ in the projective space $\mathbb{P}^{n+1}$ .  

For a smooth projective variety $X$, we define $Z _r(X)=(\mathscr{C} _r(X))^+$, the naive group completions  of  $\mathscr{C} _r(X)=\amalg _d\mathscr{C} _{r,d}(X)$. Here $\mathscr{C} _{r,d}(X)$ is the Chow variety of algebraic cycles of degree $d$ and dimension $r$. The topology on the group $Z _r(X)$ is the quotient topology induced by the complex topology of the projective varieties $\mathscr{C} _{r,d}(X)$.   We call  $Z _r(X)$ the topological space of $r-$dimensional algebraic cycles. The empty cycle $0\in\mathscr{C} _r(X)$ is the natural base point of $Z _r(X)$. We define 
   $$L _qH _n(X)=\pi _{n-2q}(Z _q(X))$$
 for any $0\leq q\leq d$ and $n\geq 2q$.  For a quasi-projective variety $U$, with a projective closure $X$, we let (see \cite{LF})
 $$Z _r(U):=Z _r(X)/Z _r(X\setminus U).$$
  This is, up to isomorphism,  a well defined object in the category of topological abelian groups that admit a structure of CW-complex with inverted homotopy equivalences  (\cite{LF}, \cite{FG}, \cite{P}). We call this category $\mathcal{H}^{-1}AbTop$. 
  
  Moreover, for a closed embedding of projective varieties $Y\subset X$, the exact sequence of topological groups
  $$0\rightarrow Z _r(Y)\rightarrow Z _r(X)\rightarrow Z _r(X\setminus Y)\rightarrow 0$$
  gives a long exact localization sequence of homotopy groups (\cite{LF}, \cite{FG}) 
  \begin{align*}
  \label{local}
  ..\rightarrow \pi _*Z _r(Y)\rightarrow \pi _*Z _r(X)&\rightarrow \pi _*Z _r(X\setminus{Y})\rightarrow \pi _{*-1}Z _r(Y)\rightarrow ..\\
    &..\rightarrow \pi _0Z _r(Y)\rightarrow \pi _0Z _r(X)\rightarrow \pi _0Z _r(X\setminus Y)\rightarrow 0.
 \end{align*}
 For $q<0$, we define the following topological cycle spaces (with the quotient topology):
$$Z _q(X)=Z _0(X\times \mathbb{A}^{-q}):=Z _0(X\times \mathbb{P}^{-q})/Z _0(X\times \mathbb{P}^{-q-1}).$$
The homotopy groups of these cycle spaces give the negative Lawson homology, i.e.
$$L _qH _n(X):=\pi _{n-2q}(Z _{q}(X))$$ 
for any $q<0$. The following equalities (\cite{FHW}) show that these groups are all isomorphic with the Borel-Moore homology of $X$. For any $q<0$
\begin{equation}
\label{qn}
L _qH _n(X)=\pi _{n-2q}(Z _0(X\times \mathbb{A}^{-q}))=H^{BM} _{n-2q}(X\times \mathbb{A}^{-q})\simeq H^{BM} _n(X^{an})=L _0H _n(X).
\end{equation}
The space of t-cocycles for a smooth projective variety $X$ is defined to be the following naive completion
$$Z^t(X)=\left( \mathcal{M}or(X,C _0(\mathbb{A}^t))^{an}/\mathcal{M}or(X,C _0(\mathbb{A}^{t-1}))^{an}\right)^+$$ 
where by $\mathcal{M}or(X,C _0(Y))^{an}$ we mean the abelian monoid of morphisms between $X$ and the Chow monoid $C _0(Y)$ provided with the compact-open topology (\cite{FL}). 

The following ``Poincare duality" type of theorem was proved in \cite{FL1}:
\begin{theorem}(\cite{FL1})
\label{dual}
There is a homotopy equivalence 
$$\mathcal{D}: Z^t(X)\rightarrow Z _d(X\times \mathbb{A}^t)\simeq Z _{d-t}(X)$$
for any smooth projective variety $X$ of dimension $d$ and for any $t\geq 0$.
\end{theorem} 
Theorem \ref{dual} says that the cycle spaces $Z _q(X)$, with $q\leq 0$, are, up to homotopy, the t-cocycle spaces $Z^t(X)$, with $t\geq dim(X)$.  In Proposition \ref{sneg} we give a simple application of this remark. 

For  a proper map of quasi-projective varieties $f:X\rightarrow Y$ and any $r\in \mathbb{Z}$, we have the push-forward map $f _*:Z _r(X)\rightarrow Z _r(Y)$ (\cite{FG}, \cite{HULi}). For any locally complete intersection map $f:X\rightarrow Y$ of codimension $d$ ($d=dim(Y)- dim(X)$), we have a well-defined Gysin map (in $\mathcal{H}^{-1}AbTop$)
$$f^*: Z _r(Y)\rightarrow Z _{r-d}(X)$$
for any $r\in \mathbb{Z}$ (\cite{FG}). The Gysin map for a proper map of quasi-projective varieties $f:X\rightarrow Y$, with $Y$ a smooth quasi-projective variety, is defined to be the composition 
$$Z _s(Y)\stackrel{pr _2^*}{\rightarrow}Z _{s+dim(X)}(X\times Y)\stackrel{\delta^*}{\rightarrow} Z _{s+dim(X)-dim(Y)}(X)$$
where $pr _2^*$ is the flat pull back of the projection map on $Y$ and $\delta^*$ is the Gysin map of the regular embedding of the graph of $f$. In the case of a regular embedding $i _V:V\hookrightarrow X$ of codimension $d$, we define the Gysin map $i _V^*$ to be the following composition (well defined in $\mathcal{H}^{-1}AbTop$):
$$i _V^*: Z _r(X)\stackrel{pr _1^*}{\rightarrow} Z _{r+1}(X\times \mathbb{A}^1)\stackrel{\delta}{\rightarrow} Z _r(N _VX)\stackrel{\pi^{*-1}}{\rightarrow}Z _{r-d}(V).$$
The left map is a flat pull-back of the projection on $X$, the middle map is the ``specialization map" given by the deformation of the normal cone $N _VX$ (\cite{Ful}, \cite{FG}) and the last map is the inverse of the flat pull-back vector bundle isomorphism $\pi^*: Z _{r-d}(V)\rightarrow Z _r(N _VX)$. 

Using these Gysin maps one can construct a well defined intersection product on cycle spaces in $\mathcal{H}^{-1}AbTop$. For example, if $i _V:V\hookrightarrow X$ is a regular embedding of an r-dimensional subvariety $V$ in the smooth quasi-projective variety $X$ of dimension d, then the intersection with $V$ on $X$ is given by the composition 
$$Z _s(X)\stackrel{i _{V}^*}{\rightarrow} Z _{r+s-d}(V) \stackrel{i _{V*}}{\rightarrow} Z _{r+s-d}(X)$$
where $i _V^*: Z _s(X){\rightarrow} Z _{s+r-d}(V)$ is the Gysin map associated to the regular embedding $V\hookrightarrow X$ and $r+s\geq d$. In general, we have the following theorem:
\begin{theorem}(\cite{FG}, Theorem 3.5)
If $X$ is a smooth quasi-projective variety of dimension $d$ and if $r+s\geq d$ then there is an intersection pairing (in $\mathcal{H}^{-1}AbTop$)
$$Z _r(X)\tensor Z _s(X)\rightarrow Z _{r+s-d}(X)$$
which on $\pi _0$ gives the usual intersection pairing on Chow groups of algebraic cycles modulo algebraic equivalence.
\end{theorem}

According to (\cite{FM2}, \cite{LF}), there is an operation on Lawson groups, called s-map
$$s:L _rH _n(X)\rightarrow L _{r-1}H _n(X) $$
for any quasi-projective variety $X$.   The inverse of the isomorphism  (\ref{qn}) is given in the following proposition.
\begin{proposition}
\label{sneg}
The s-map on the negative Lawson homology of a smooth projective variety $X$ of dimension $d$ is an isomorphism, i.e.
$$L _0H _n(X)\stackrel{s}{\simeq} L _{-1}H _n(X)\stackrel{s}{\simeq} L _{-2}H _n(X)\stackrel{s}{\simeq}... $$
\end{proposition}
\begin{proof}
According to (\cite{FL1}, Proposition 2.6), for any smooth projective variety $X$ of dimension $d$ and for any $t\geq 0$ we have the following commutative diagram in ($\mathcal{H}^{-1}AbTop$):
$$\begin{CD}
      Z^t(X)\wedge S^2 @>\stackrel{\mathcal{D}\wedge 1}{\simeq}>> Z _d(X\times \mathbb{A}^t)\wedge S^2\\
                     @V s VV         @VV s V\\
      Z^{t+1}(X)  @>\stackrel{\mathcal{D}}{\simeq}>> Z _{d-1}(X\times \mathbb{A}^t).\\
\end{CD}$$
If $t\geq d$, then the left vertical arrow is a homotopy equivalence (\cite{FL1}, Theorem 5.8). This implies that the right vertical s-map is a homotopy equivalence. Then we have
$$\pi _n(Z  _0(X))\stackrel{s}{\simeq} \pi _{n+2}(Z _{-1}(X))=L _{-1}H _n(X)\stackrel{s}{\simeq} \pi _{n+4}(Z _{-2}(X))=L _{-2}H _n(X)\stackrel{s}{\simeq}...$$
for any $n\geq 0$.
\end{proof}
The s-map is a natural map that commutes with push-forwards, flat pull-backs, Gysin maps, intersection with a cycle and with localization sequences (\cite{Fil}, Page 4 and Proposition 1.7;  \cite{FL1}, Proposition 2.3).

We let $cyc _{p,q}$ denote the (generalized) cycle maps
$$cyc _{p,q}:L _pH _q(X)\rightarrow H _q^{BM}(X^{an})$$
for any quasi-projective variety $X$. We recall that if  $X$ is a smooth projective variety then $H _q^{BM}(X^{an})=H _q(X^{an})$. The generalized cycle maps are compositions of s-maps (\cite{FM2}, \cite{Fil}), i.e.
$$cyc _{p,q}:L _pH _q(X)\stackrel{s}{\rightarrow} L _{p-1}H _q(X)\stackrel {s}{\rightarrow}...\stackrel{s}{\rightarrow}L _1H _q(X)\stackrel{s}{\rightarrow}L _0H _n(X)=H _n^{BM}(X^{an}) $$ 
for any $q\geq 2p\geq 0.$ If $p<0$, then the maps $cyc _{p,q}$ are isomorphisms (Proposition \ref{sneg} and the isomorphism (\ref{qn})).

The above decomposition gives a filtration on the kernel of $cyc _{p,q}$.  The kernel of $cyc _{r,2r}$ is the Griffiths group of algebraic r-cycles \cite{Fr}. Let 
$$Z _r(X)\stackrel{\pi}{\rightarrow} \ L _rH _{2r}(X)=\pi _0(Z _r(X))\stackrel{s}{\rightarrow} L _{r-1}H _{2r}(X)\stackrel{s}{\rightarrow}...\stackrel{s}{\rightarrow}H _{2r}(X).$$
Define $S _iZ _r(X)=Ker(s^i\circ \pi)$. We know that 
$$S _0Z _r(X)=\textnormal{\{algebraic cycles algebraically equivalent to zero\}}$$ and that $$S _rZ _r(X)=\textnormal{\{algebraic cycles homologically equivalent to zero\}}.$$ If we take the quotient of the above filtration by $Ker (\pi)$, we obtain a filtration on the Griff$ _r(X)$ called the s-filtration of the Griffiths groups. This is
\begin{equation}
\label{s-filtration}
0\subset S _1Z _r(X)/S _0Z _r(X)\subset ..\subset \textnormal{Griff} _r(X)=S _rZ _r(X)/S _0Z _r(X).
\end{equation}

We let  $L^{hom} _pH _q(X)=Ker(cyc _{p,q})$ and $C _{p,q}(X)=Coker (cyc _{p,q})$ to be the kernel and the cokernel of the maps $cyc _{p,q}$ .  

The following proposition was proved in \cite{Voin}:

\begin{proposition}
\label{propo}(\cite{Voin})

Let $X$ be a smooth projective variety of dimension $d$. Assume that the Bloch-Kato conjecture is valid for all the primes.  Then:     

a) Let $n\geq d+q-1$. Then $ L^{hom} _qH _n(X)$ is divisible and $C _{q,n}(X)$ is torsion free. 

b) $L _qH _n(X)$ is uniquely divisible for $n>2d$ and $L _qH _{2d}(X)$ is torsion free (for any $q\leq d)$. 

\end{proposition}
The following theorem is the projective bundle theorem in Lawson homology.
\begin{theorem} (\cite{FG},\cite{Hu})
\label{projb}
Let $E$ be a rank $r+1$ vector bundle over a quasi-projective variety $Y$, let $p:P(E)\rightarrow Y$  be the canonical map and  let $O _{P(E)}(1)$ be the canonical line bundle on $P(E)$. Let $h=c _1(O _{P(E)}(1))$. Then
$$\phi=\sum _{j=0}^rh^{r-j}\circ p^*:\oplus _{j=0}^{r}Z _{i-j}(Y)\rightarrow Z _i(P(E))$$
is a homotopy equivalence for any $i\geq 0$. In particular
$$\phi=\sum _{j=0}^rh^{r-j}\circ p^*:\oplus _{j=0}^{r}\pi _k Z _{i-j}(Y)\rightarrow \pi _k Z _i(P(E))$$
is an isomorphism for any $k\geq 0$, $i\geq 0$

\end {theorem}

 We will refer later in this paper to the following two conjectures.
\begin {conjecture} (Suslin's conjecture)
\label{sc}
The generalized cycle map
$$cyc _{q,n}:L _qH _n(X)\rightarrow H _n(X^{an})$$
is an isomorphism for $n\geq d+q$ and a monomorphism for $n\geq d+q-1$ for any smooth projective variety $X$ of dimension $d$.
\end{conjecture}
We notice that this conjecture contains a conjecture due to E. Friedlander and B. Mazur \cite{FM2}.

\begin{conjecture}(Friedlander-Mazur conjecture)
For any complex smooth projective variety $X$ of dimension $d$
$$L _qH _n(X)=0$$
for any $n>2d$.
\end{conjecture}

These highly non-trivial conjectures have been checked on certain varieties, like curves, surfaces, rationally connected threefolds and fourfolds, smooth projective toric varieties (\cite{FHW}, \cite{Voin}).

The (singular) semi-topological K-theory of a complex projective variety $X$ was introduced in \cite{FW}. This is defined by
$$K^{sst} _*(X)=\pi _*(Mor(X,Grass)^+)$$
where $Grass=\amalg _{n,N}Grass _{n}(\mathbb{P}^N)$. By $Mor(X,Grass)^+$, we define the topological group given by the homotopy completion of the space of algebraic maps between $X$ and $Grass$.  The main tool for computing $K^{sst} _*(X)$ is a spectral sequence with $E^2-$term given  by the Lawson homology of $X$, when $X$ is a smooth variety.
\begin{theorem}(\cite{FHW})
\label{ss}
For any smooth, projective complex variety $X$ and any abelian group $A$, there is a natural map of strongly convergent spectral sequences

\begin{equation*}
\xymatrix{
E^{p,q}_2(sst)=L^{-q}H^{p-q}(X,A) & {\Longrightarrow}\ar[d] & K^{sst} _{-p-q}(X,A) \\
E^{p,q}_2(top)=H^{p-q}(X^{an},A) & {\Longrightarrow} & ku^{p+q}(X^{an},A). 
}
\end{equation*} 
inducing the usual maps on both $E _2-$terms and abutments.
\end{theorem}

\section{About the Fano variety of k-planes}

In this section, we describe a geometrical construction, introduced in \cite{Le}, for generic hypersurfaces $X$ with ``enough" $k-$planes  and prove Theorem \ref{the1} which will be the main tool in the fourth section.

We start this section with the definition of the Fano variety of k-planes on a
projective variety.
\begin{definition}
Let $Z\subset \mathbb{P}^N$ be a variety. We let 

$$\Omega _Z(k)=\{\mathbb{P}^{k's}\subset Z\}\subset Grass _k\mathbb{P}^N$$ 
be the set of all k-planes included in Z and call it the Fano variety of k-planes of Z.
\end{definition} 
The following theorem describes $\Omega _X(k)$ for  a generic hypersurface $X\subset \mathbb{P}^{n+1}$ of degree $d\leq n+1$.
\begin{theorem}(Borcea \cite{Bor}, Corollary 2.2)
\label{t101}

Let $X\subset \mathbb{P}^{n+1}$ be a generic hypersurface of degree $d\leq n+1$ and
let $k=[\frac{n+1}{d}]$. Then $\Omega _X(k)$ is non-empty and smooth of pure
dimension $$\gamma=(k+1)(n+1-k)-\dbinom{d+k}{k}$$ provided that  $\gamma\geq 0$ and
$X$ is not a quadric. In the case $X$ is a quadric we require $n\geq 2k$.
Furthermore, if $\gamma>0$ or if in the case $X$ a quadric with $n>2k$, then $\Omega
_X(k)$ is irreducible. 
\end{theorem}
 We will study below only generic hypersurfaces $X\subset \mathbb{P}^{n+1}$ of
degree $3\leq d\leq n+1$ with the property that 
\begin{equation}
\label{eq11}
dim(\Omega _X(k))=(k+1)(n+1-k)-\dbinom{d+k}{k}\geq n-2k
\end{equation}
 where $k=[\frac{n+1}{d}]$. Theorem \ref{t101} says that such varieties have $\Omega _X(k)$ non-empty, smooth and
irreducible. Consider  the homogeneous polynomial $F(X _0, X _1,.., X _{n+1})$ of
degree $d$ that gives our $X$. J. Lewis (\cite{Le}, \cite{Le1}) introduced the
following construction: let $G(X _0, X _1,.., X _{n+2})=X^d _{n+2}+F(X _0, X _1,..,
X _{n+1})$ be a homogeneous polynomial of degree $d$ in $n+3$ variables that defines
a smooth hypersurface $Z\subset\mathbb{P}^{n+2}$. We notice that $X$ is a
hyperplane section of $Z$ if we consider the embedding 
$$\mathbb{P}^{n+1}=V(X _{n+2})\subset\mathbb{P}^{n+2}$$ 
where by  V(H) we understand the set of zeros of the homogenous polynomial H. Moreover $\Omega _Z(k)$ is smooth, irreducible and
of dimension $(n-k+1)+l$ (\cite{Le}), where $l=\gamma -(n-2k)\geq 0$.

 Let $\Omega _Z$ be the subvariety cut out by $l$ general hyperplane sections of the projective variety $\Omega _Z(k)$ in a projective embedding and let $\Omega _X=\Omega _X(k)\cap \Omega _Z$. 
 Using Bertini's theorem we conclude that $\Omega _Z$ and $\Omega _X$ are smooth varieties of pure dimension ($\Omega _Z$ is an irreducible variety) with  $dim(\Omega _Z)=n-k+1$ and $dim(\Omega _X)=n-2k$. 

Below we will recall the main known properties of these varieties and
fix some notations that we use later in the text (we use, for consistency, the same notations as in \cite{Le}, \cite{Le1}, \cite{BM}).

Let $\pi _X: P(X)\rightarrow X$ and $\pi _Z:P(Z)\rightarrow Z$ the projections from
the incidence varieties $P(X)=\{(c,x)\in \Omega _X\times X$ s.t. $x\in \mathbb{P}^k
_c\}$, $P(Z)=\{(c,x)\in \Omega _Z\times Z$ s.t. $ x\in \mathbb{P}^k _c\}$ to the
variety $X$, respectively $Z$. 

We denote by $\rho _Z:P(Z)\rightarrow \Omega _Z$ and
 $\rho _X:P(X)\rightarrow \Omega _X$ the natural projections. We have a cartesian
diagram
$$\begin{CD}
      X'=X\times _Z P(Z) @>j _2>> P(Z)\\
                     @V \pi VV         @VV \pi _Z V\\
       X@>j>> Z.\\
\end{CD}$$
If $k=1$, then $X'=Bl _{\Omega _X}(\Omega _Z)$, the blow-up of $\Omega _X\subset\Omega _Z$ (see \cite{Le1}). Denote $\rho=\rho _Z\circ j _2:X'\rightarrow \Omega _Z$. We also have a commutative
diagram 

$$\begin{CD}
      P(X) @>i>> X'\\
                     @V \rho _X VV         @VV \rho V\\
       \Omega _X @>j _0>> \Omega _Z\\
\end{CD}$$
with the embedding maps $j _0, i$. We have the following properties (\cite{Le1}, \cite{Le}):

1) $\pi$ and $\pi _Z$ are generically finite.

2) $\rho _X:P(X)\rightarrow \Omega _X$ and $\rho _Z:P(Z)\rightarrow \Omega _Z$ are
$\mathbb{P}^k$-bundles.

3) $X'$ is smooth by Bertini's theorem.

4) $\rho  _{Z/X}:=\rho _{| _{X' \setminus P(X)}}:X' \setminus P(X)\rightarrow \Omega _Z\setminus \Omega _X$ is a
$\mathbb{P}^{k-1}$-bundle.

5) $dim(X)=dim(X')=n$, $dim(Z)=dim(P(Z))=n+1$, $dim(P(X))=n-k$ and all varieties are
smooth. 

  Let $H _Z=\mathbb{P}^{n+1}\cap Z$ and $H _X=H _Z\cap X$ be generic hyperplane
sections of Z, respectively X. Then we denote $\mu=\pi^{-1}(H _X)$ and $\mu
_Z=\pi^{-1} _Z(H _Z)$ and we use the same notations for their algebraic cycle classes. 

The following definition introduces one of the objects of study in this paper. In the above notation, we have 
\begin{definition} 
\label{cyl}
 We call the map 
$$\pi _*\circ i _*\circ \rho _X^*:L _{r-k}H _{*+2(r-k)}(\Omega _X) \rightarrow L _rH _{*+2r}(X)$$
 the cylindrical homomorphism on  the Lawson homology groups of $X$. 
\end{definition}
For a motivation behind the Definition \ref{cyl} see \cite{HC}. For a related discussion, see the remarks after the proof of Theorem \ref{t4}.
\begin{remark} (\cite{HC})

 Cylindrical homomorphism (Definition \ref{cyl}) coincides with the action of the incidence correspondence in $\Omega _X\times X$ on the Lawson homology.  
\end{remark}

\begin{theorem}
\label{the1}
 With notations as above, there is a homotopy equivalence
$$Z _r(X')\simeq \oplus _{j=0}^{k-1}Z _{r-j}(\Omega _Z)\oplus Z _{r-k}(\Omega _X)$$
for any $r\geq 0$.
\end{theorem} 
\begin{proof}
We prove this theorem following the lines of the proof of (\cite{Le}, Proposition 3.2).  We will only focus on the parts of the proof that require an adaptation to our topological groups.

In Appendix,  Proposition \ref{Pu}, we prove that $$(\rho\times 1) _{*}\circ pr _1^*\mu^{k-1-l} \circ(\rho\times
1)^*:Z _r(\Omega _Z\times \mathbb{P}^m)\rightarrow Z _{r+l}(\Omega _Z \times
\mathbb{P}^m)$$
 is the zero map for any $0<l\leq k-1$ and the identity map for $l=0$. Here $pr _1: X'\times \mathbb{P}^m\rightarrow X'$ is the canonical projection map on $X'$ and $m\geq 0$.

For any $r,m\geq 0$, we have the following commutative diagram of fibration sequences:

\begin{equation*}
\begin{CD}
      Z _r(\Omega _Z\times \mathbb{P}^{m-1}) @>>> Z _r(\Omega _Z\times \mathbb{P}^m)  @>>>Z
_r(\Omega _Z\times \mathbb{A}^m)\\
                     @V{g}VV         @VV{g}V  @VV{g}V\\
      Z _{r+l}(\Omega _Z\times \mathbb{P}^{m-1})  @>>> Z _{r+l}(\Omega _Z\times \mathbb{P}^m) @>>>Z
_{r+l}(\Omega _Z\times \mathbb{A}^m)
\end{CD}
\end{equation*}

where $g=(\rho\times 1) _{*}\circ pr _1^*\mu^{k-1-l}\circ(\rho\times 1)^*$. This implies that the induced map
$$g=(\rho\times 1) _{*}\circ pr^* _1\mu^{k-1-l}\circ(\rho\times 1)^*:Z _r(\Omega
_Z\times \mathbb{A}^m)\rightarrow Z _{r+l}(\Omega _Z\times \mathbb{A}^m)$$
is zero for any $r,m \geq 0$ and  $0<l\leq k-1$ and it is identity for any $r,m \geq 0$ and
$l=0$. We conclude that for any $r\in \mathbb{Z}$, the map 
$$\rho _{*}\circ \mu^{k-1-l}\circ\rho^*: Z _r(\Omega _Z)\rightarrow Z _{r+l}(\Omega _Z)$$
is zero for $l>0$ and is identity for $l=0$.

 Let's consider the following homomorphisms:

$$\lambda _1=\oplus _{l=0}^{k-1} \rho _*\circ \mu^{k-1-l}:\pi _*Z _r(X')\rightarrow
\oplus _{l=0}^{k-1} \pi _*Z _{r-k+1+l}(\Omega _Z)$$   
and 
$$\lambda _2=\sum _{l=0}^{k-1}\mu^l\circ \rho^*:\oplus _{l=0}^{k-1}\pi _*Z
_{r-k+1+l}(\Omega _Z)\rightarrow \pi _*Z _r(X')$$
and $$T=\lambda _1\circ \lambda _2: \oplus _{l=0}^{k-1}\pi _*Z
_{r-k+1+l}(\Omega _Z)\rightarrow \oplus _{l=0}^{k-1}\pi _*Z
_{r-k+1+l}(\Omega _Z).$$ 

We can prove that $\lambda _1$ is surjective and $\lambda _2$ is injective. To see this consider
$\xi=(*,*,..,*)\in \oplus _{l=0}^{k-1} \pi _*Z _{r-k+1+l}(\Omega _Z)$. Then $T(\xi)-\xi=(*,*,..,*,0)$. This is because the last component of $T$ is 
$$\rho _*\circ(\sum _{l=0}^{k-1}\mu^l\circ\rho^*)=\sum _{l=0}^{k-1}\rho _*\circ\mu^l\circ\rho^*=\textrm{identity}$$
according to the above results.
It implies that 
$$(T-I)^k=T^k+a _{k-1}T^{k-1}+...+a _1T\pm I=0$$
 on $\oplus _{l=0}^{k-1} \pi _*Z _{r-k+1+l}(\Omega _Z)$. 

This shows that there is a polynomial $f\in\mathbb{Z}[X]$ such that $f(T)T=Tf(T)=\pm I$ on
$\oplus _{l=0}^{k-1} \pi _*Z _{r-k+1+l}(\Omega _Z)$. This implies that
$\lambda _1$ is surjective and $\lambda _2$ is injective. 

We have the following commutative diagram with exact rows:
\begin{equation*}
\xymatrix@-1.3pc{
 \ar[r]&\pi _{*+1}(Z _r(P(X))) \ar[r]^{i _*}\ar[d]^{\alpha}   &   \pi _{*+1}(Z _r(X')) \ar[r]\ar[d]^{(1)} _{\mp f(T)\circ \lambda _1}   &  \pi _{*+1}(Z _r(X' {\setminus P(X)}))  \ar[r]\ar[d]^{(2)} _{\beta\simeq}& \\ 
\ar[r]&   \oplus _{l=0}^{k-1} \pi _{*+1}Z _{r-k+1+l}(\Omega _X) \ar[r]  &  \oplus _{l=0}^{k-1} \pi _{*+1}Z _{r-k+1+l}(\Omega _Z)  \ar[r]  & \oplus _{l=0}^{k-1} \pi _{*+1}Z _{r-k+1+l}(\Omega _Z\setminus\Omega _X)\ar[r]& \\
}
\end{equation*}
\begin{equation}
\xymatrix@-1.2pc{
 \ar[r]& \pi _*(Z _r(P(X)))\ar[r]^{i _*}\ar[d] _{\alpha} ^{(3)} & \pi _*(Z _r(X'))\ar[r]\ar[d] _{\mp f(T)\circ \lambda _1} ^{(4)} &  \pi _*(Z _r(X'\setminus P(X)))\ar[r] \ar[d] _{\beta} ^{\simeq} &\\
\ar[r]& \oplus _{l=0}^{k-1} \pi _*Z _{r-k+1+l}(\Omega _X) \ar[r]  &  \oplus _{l=0}^{k-1} \pi _*Z _{r-k+1+l}(\Omega _Z)  \ar[r]  & \oplus _{l=0}^{k-1} \pi _*Z _{r-k+1+l}(\Omega _Z\setminus\Omega _X)\ar[r]& \\
}
\label{dia2}
\end{equation}

The exact rows are given by the long exact localization sequences  applied to the closed embeddings
$i:P(X)\hookrightarrow X'$ and $j _0:\Omega _X\hookrightarrow \Omega _Z$. 

The vertical rows are constructed from the maps given by the projective bundle theorem (Theorem \ref{projb})
applied to the $\mathbb{P}^k$-bundle $\rho _X :P(X)\rightarrow \Omega _X$ and to the $\mathbb{P}^{k-1}$-bundle $\rho _{Z/X}:  X'\setminus P(X)\rightarrow \Omega _Z\setminus \Omega _X$.

In Appendix, Proposition \ref{dchas}, a diagram chase in Diagram (\ref{dia2}) gives us the following isomorphism
\begin{equation}
\label{lis}
\lambda _2+i _*\circ \rho^* _X:\oplus _{l=0}^{k-1} \pi _*Z _{r-k+1+l}(\Omega
_Z)\oplus \pi _*Z _{r-k}(\Omega _X)\simeq\pi _*(Z _r(X'))
\end{equation}
for any $r\geq 0$. We notice that the map $\lambda _2+i _*\circ \rho^*
_X$ could be defined on the cycle spaces. We may rewrite the indices
$(j=k-1-l)$ such that $\oplus _{l=0}^{k-1} \pi _*Z _{r-k+1+l}(\Omega _Z)=\oplus
_{j=0}^{k-1} \pi _*Z _{r-j}(\Omega _Z)$.

Because the topological groups $Z _p(X)$ are C.W. complexes for any projective
variety $X$ \cite{F}, the isomorphism (\ref{lis}) gives us a homotopy equivalence
$$Z _r(X')\simeq \oplus _{j=0..k-1}Z _{r-j}(\Omega _Z)\oplus Z _{r-k}(\Omega _X)$$
for any $r\geq 0$.
\end{proof}
In particular, for $r=0$, the isomorphism (\ref{lis}) gives the decomposition of
the singular homology of $X'$ proven by J. Lewis in (\cite {Le}, Proposition 3.2).
\begin{corollary}(\cite{Le})
 There is an isomorphism 
 $$ \{\oplus _{l=0}^{k-1}H^{n-2l}(\Omega _Z,\mathbb{Z})\}\oplus H^{n-2k}(\Omega _X,\mathbb{Z})\simeq H^n(X',\mathbb{Z}).$$
\end{corollary}

Moreover the s-map respects the decomposition (\ref{lis}) because it commutes with the
maps involved in the above decomposition (see Proposition \ref{sneg} and the discussion after Proposition \ref{sneg}). This means that the cycle maps 
 $$cyc _{r,*+2r}=s^r: L _rH _{*+2r}(X')\rightarrow H _{*+2r}(X')$$
  respect the decomposition (\ref{lis}). This gives the next corollary.

\begin{corollary}
\label{c1}
For any $r\geq 0$, the isomorphism
$$\lambda _2+i _*\circ \rho^* _X:\oplus _{l=0}^{k-1} \pi _*Z _{r-k+1+l}(\Omega _Z)\oplus \pi _*Z _{r-k}(\Omega
_X)\simeq\pi _*(Z _r(X'))$$
 commutes with the s-map operation. In particular, the kernel and the cokernel of the cycle maps from Lawson homology of $X'$ to the singular
homology of $X'$ have the decompositions 
$$L^{hom} _rH _{*+2r}(X')\simeq \oplus _{l=0}^{k-1}L^{hom} _{r-k+l+1}H
_{*+2r-2k+2l+2}(\Omega _Z)\oplus L^{hom} _{r-k}H _{*+2r-2k}(\Omega _X).$$
and 
$$C _{r,*+2r}(X')\simeq \oplus _{l=0}^{k-1}C _{r-k+l+1,*+2r-2k+2l+2}(\Omega
_Z)\oplus C _{r-k,*+2r-2k}(\Omega _X).$$
\end{corollary}

Corollary \ref{c1} will be essential in proving  certain properties of the cylindrical homomorphism on Lawson homology (see Theorem \ref{t3} and Theorem \ref{t4}).

\section{Some remarks on the Lawson homology groups of a generic hypersurface}
 In this section we extend to Lawson homology groups a weak Lefschetz result (\cite{Le1}, Proposition 2.1) proved on Chow groups and discuss some applications of the classical weak Lefschetz theorem on Lawson homology (Proposition \ref{p1}). The results in this section will be used as technical tools in the next section.
 
 Let $X\subset \mathbb{P}^{n+1}$ be a nonsingular
hypersurface given by a homogeneous polynomial $F(X _0, X _1,.., X _{n+1})$ of
degree $d$. Consider 
$$G(X _0, X _1,.., X _n, X _{n+1}, X _{n+2})=X^d _{n+2}+F(X _0, X _1,.., X _{n+1})$$ 
a homogeneous polynomial of degree $d$. Then $G$ defines a smooth hypersurface $Z\subset
\mathbb{P}^{n+2}$ with the property that $Z\cap \mathbb{P}^{n+1}=X$, where we
considered the embedding $\mathbb{P}^{n+1} =V(X _{n+2})\subset \mathbb{P}^{n+2}$.
Consider the map $v _p:\mathbb{P}^{n+2}\rightarrow \mathbb{P}^{n+1}$, the projection
from the point $p=[0,0,...,0,1]$ and the inclusions $k:
X\hookrightarrow\mathbb{P}^{n+1}$ and $j:X\hookrightarrow Z$. It is obvious that $v
_p\circ j=k$, i.e.
$$k:X\stackrel{j}{\hookrightarrow} Z\stackrel{v _p}{\rightarrow} \mathbb{P}^{n+1}.$$

Consider the graphs of these maps $W _j, W _k, W _{v _p}$ as correspondences in their
corresponding Chow groups. The graphs have
the property (see \cite{Le}) that
$$d ^tW _j= ^tW _k\circ W _ {v _p}\in CH _n(Z\times X)$$
where on the right side of the equality we have a composition of correspondences. We may view this equality in $A _n(Z\times X)$, the Chow group of algebraic cycles modulo algebraic equivalence.

The action of correspondences and its properties [as in (\cite{Ful}, Chapter 16)] can be word for word extended on Lawson homology because of the good properties this theory enjoys (see \cite{HULi}). 
\begin{theorem}
\label{the2}
Let $X\subset \mathbb{P}^{n+1}$ be a nonsingular hypersurface of degree $d$ and
$Z\subset \mathbb{P}^{n+2}$ be the associated smooth hypersurface defined  above. Let $m\geq 2p\geq 0$. Consider
 the Gysin map $$j^*: L _pH _m(Z)\rightarrow L _{p-1}H _{m-2}(X)$$ defined by the regular embedding $j:X\hookrightarrow Z$. Then $dj^*$ is zero on the kernel and on the cokernel of the cycle map
$L _pH _m(Z)\rightarrow H _m(Z^{an})$.
\end{theorem}
\begin{proof}
 Consider the maps $j,  k, v _p$ defined as above and the relation between the
correspondences given by the graphs of these maps: $d ^tW _j= ^tW _k\circ W _ {v
_p}$. The action of these correspondences on the Lawson homology groups gives the equality
$$dj^*=k^*v _{p*}.$$ The fact that the action of the correspondence $^tW _j$  has the same effect
as the Gysin map $j^*$ is a direct consequence of the projection formula and 
base change formula for Gysin maps proved in \cite{P} (for the formalism see \cite{Ful} Proposition 16.1.2 and Proposition 16.1.1 b),c)). 

Applying the s-map $p$ times gives the following commutative diagram:
 $$\begin{CD}
      L _pH _m(Z) @>v _{p*}>> L _pH _m(\mathbb{P}^{n+1}) @>k^*>> L _{p-1}H _{m-2}(X)\\
            @Vcyc _{p,m}VV            @Vcyc _{p,m}VV         @Vcyc _{p-1,m-2}VV\\
      H _m(Z^{an}) @>v _{p*}>> H _m(\mathbb{P}^{n+1}) @>k^*>> H _{m-2}(X^{an})\\
\end{CD}$$
for any $p, m\geq 0$, $m\geq 2p$. The middle cycle map of the above diagram is an isomorphism for
any $p, m\geq 0$ \cite{L}. The composition of the horizontal maps is given by the
Gysin map $dj^*$. It implies that $dj^*(L^{hom} _pH _m(Z))=0=dj^*(C _{p,m}(Z))$.
\end{proof}
\begin{corollary}
\label{c2}
Consider the Gysin map $j^*:H _m(Z^{an})\rightarrow H _{m-2}(X^{an})$ and fix $p>0$. Then for
any $x\notin Im(cyc _{p,m})$ we have $dj^*(x)\in Im(cyc _{p-1,m-2})$. 
\end{corollary}
\begin{proof}
This is just a reformulation of the statement  $dj^*(C _{p,m}(Z))=0$ proved above.
\end{proof}
The classical weak Lefschetz theorem, together with the Poincare duality theorem \cite{Vo}, says that a smooth hypersurface
$X\subset \mathbb{P}^{n+1}$ has the following cohomology: 

-$H^k(X,\mathbb{Z})=0$, for k odd, $k\neq dim(X)$; 

-$H^{2k}(X,\mathbb{Z})\simeq \mathbb{Z}\alpha$ for $2k>dim(X)$. Here $\alpha$ is a class of a topological cycle with the intersection $<\alpha,h^{n-1-k}>=1$. By $h$ we denote the algebraic class of a hyperplane section of $X$. Moreover, there is a positive integer $a\in \mathbb{N}$ such that $a\alpha$ is an algebraic cycle. 

-$H^{2k}(X,\mathbb{Z})\simeq \mathbb{Z} h^k$ for $2k<dim(X)$, with  $h$ being the algebraic class of a hyperplane section of $X$.

If $X\subset\mathbb{P}^{n+1}$ is a smooth generic hypersurface of small degree $3\leq d\leq n+1$ that fulfills the condition $dim(\Omega _X(k))\geq 0$ or equivalently $(k+1)(n+1-k)-\dbinom{d+k}{k}\geq n-2k$, with $k=[\frac{n+1}{d}]$, then we have at least one k-plane included in $X$ \cite{Bor}. This implies that $H^{2n-2i}(X,\mathbb{Z})=\mathbb{Z}\alpha = \mathbb{Z}$ for any $0\leq i\leq k$,  because we may take the generator $\alpha$ to be an $i-$plane included in $X$. We remark that a generic rationally connected hypersurface has always a line included in it. 

We can conclude now the following proposition:

\begin{proposition}
\label{p1}
For any smooth generic hypersurface $X\subset \mathbb{P}^{n+1}$ of degree $3\leq d\leq n+1$, the
cycle maps
$$cyc _{p,m}:L _pH _m(X)\rightarrow H _m(X^{an})$$ are 

1) rationally surjective for any $m\neq dim(X)$.

2) surjective with integer coefficients for any $m\neq dim(X)$, m odd.

3) surjective with integer coefficients for any $m=2r>dim(X)$.

4) surjective with integer coefficients for any $0<m=2r\leq 2k$ if $dim(\Omega _X(k))\geq 0$.
\end{proposition}
We notice that the first three points of the Proposition \ref{p1} apply to any smooth projective hypersurface. We will need later the following corollary:
\begin{corollary}
\label{coro2}
Let $X$ be a smooth generic projective cubic hypersurface of dimension $n=5,6$ or $8$. Then the cycle maps 
$$cyc _{p,m}:L _pH _m(X)\rightarrow H _m(X^{an})$$
are surjective for any $m\neq dim(X)$.
\end{corollary}  
\begin{proof}
If $n=5$ or $6$ then $k=[\frac{n+1}{d}]=2$. Proposition \ref{p1} and Theorem \ref{t101} imply that all the cycle maps $cyc _{p,m}$, with $m\neq dim(X)$, are surjective. For a generic cubic eightfold, we have $k=3$ and again the conclusion follows from Propositions \ref{p1} and Theorem \ref{t101}. 
\end{proof}
\section{Applications}
 In this section we will show some applications of the results proved in the previous sections. Our main theorems are Theorem \ref{t3} and Theorem \ref{t4}.

We proved in Theorem \ref{the1} that we have the following isomorphism:
$$\lambda _2+i _*\circ \rho^* _X:\oplus _{l=0}^{k-1} \pi _*Z _{r-k+1+l}(\Omega
_Z)\oplus \pi _*Z _{r-k}(\Omega _X)\simeq\pi _*(Z _r(X')).$$
Because  $\pi:X'\rightarrow X$ is a generically finite map  we have
\begin{equation}
\label{fm}
\pi _*\circ\pi^*=(deg\pi)id:L _*H_*(X)\rightarrow L _*H _*(X)
\end{equation}
as a result  of the projection formula for Gysin maps proved by C. Peters in \cite{P} (see
also \cite{FL}). This implies that 
$\pi _*\tensor\mathbb{Q}:L _*H _*(X') _\mathbb{Q}\rightarrow L _*H _*(X) _\mathbb{Q}$ is a surjective map. From this and the corresponding in singular homology of the Equality (\ref{fm}), we can conclude that the  restriction 
$$\pi _*\tensor\mathbb{Q}:L _*^{hom}H _*(X') _\mathbb{Q}\rightarrow L _*^{hom}H _*(X) _\mathbb{Q}$$ is a surjective map.

Then the composition  $$\pi _*\circ(\lambda _2+i
_*\circ\rho^* _X)=\pi _*\circ\lambda _2+\pi _*\circ i _*\circ \rho _X^*:\oplus _{l=0}^{k-1} \pi _*Z _{r-k+1+l}(\Omega
_Z)\oplus \pi _*Z _{r-k}(\Omega _X)\rightarrow\pi _*Z _r(X) $$ is a rationally surjective map. Using Corollary \ref{c1}, we conclude that

\begin{equation}
\label{fer}
\pi _*\circ(\lambda _2+i _{*}\circ\rho^* _X):\oplus _{l=0}^{k-1}L^{hom}
_{r-k+l+1}H _{*+2r-2k+2l+2}(\Omega _Z) _\mathbb{Q}\oplus L^{hom} _{r-k}H
_{*+2r-2k}(\Omega _X) _\mathbb{Q}\rightarrow L^{hom} _rH _{*+2r}(X) _
\mathbb{Q}
\end{equation}
is a surjective map.

 The image of the map $\pi _*\circ\lambda _2:\oplus _{l=0}^{k-1} \pi _*Z _{r-k+1+l}(\Omega
_Z)\oplus \pi _*Z _{r-k}(\Omega _X)\rightarrow Z _r(X)$ is included in the image of $j^*:Z _{r+1}(Z)\rightarrow Z _r(X)$ because
\begin{equation}
\label{e1}
\pi _*\circ\lambda _2=\pi _*\circ j^* _2(\sum^{k-1} _{l=0}\mu^l _Z\circ \rho^* _Z)=
j^*\circ\pi _{Z*}\circ (\sum _{l=0}^{k-1}\mu^l _Z\circ\rho^* _Z)
\end{equation}
 using the base change in the cartesian diagram
$$\begin{CD}
       X' @>j _2>> P(Z)\\
                     @V\pi VV         @VV\pi _ZV\\
       X @>j>> Z.\\
\end{CD}$$
 

The following theorem is one of the main theorems of this section:
\begin{theorem}
\label{t3}
Let  $X\subset \mathbb{P}^{n+1}$ be a generic hypersurface of degree $3\leq d\leq n+1$ with $dim(\Omega _X(k))\geq n-2k$. Then, the restriction of the cylindrical homomorphism
$$\pi _*\circ i _*\circ \rho _X^*:L^{hom} _{r-k}H _{*+2(r-k)}(\Omega _X)
_\mathbb{Q}\rightarrow L^{hom} _rH _{*+2r}(X) _\mathbb{Q}$$
is surjective for any $r\geq 0$. In particular 
$$L^{hom} _rH _{*+2r}(X) _\mathbb{Q}=0$$
for any $r\leq k$ and any $r\geq n-k-1$. Moreover, we may fix a constant nonzero $N$
such that $NL^{hom} _rH _{*+2r}(X)=0$ for such $r$.
\end{theorem}
\begin{proof} According to (\ref{fer})
$$\pi _*\circ(\lambda _2+i _{*}\circ\rho^* _X):\oplus _{l=0}^{k-1}L^{hom}
_{r-k+l+1}H _{*+2r-2k+2l+2}(\Omega _Z) _\mathbb{Q}\oplus L^{hom} _{r-k}H
_{*+2r-2k}(\Omega _X) _\mathbb{Q}\rightarrow L^{hom} _rH _{*+2r}(X) _
\mathbb{Q}$$
is a surjective map. From the Equality (\ref{e1}) we conclude that the composition
$$L^{hom} _{r-k}H _{*+2r-2k}(\Omega _X) _\mathbb{Q}\stackrel{\pi _*\circ i _*\circ \rho^* _X}{\rightarrow} L^{hom} _rH _{*+2r}(X)
_\mathbb{Q}\rightarrow L^{hom} _rH _{*+2r}(X) _\mathbb{Q}/Im(j^*) _\mathbb{Q}$$
is a surjective map. But Theorem \ref{the2} shows that $dj^*$ is the zero map on the
restriction to the kernel of $cyc _{q,n}$ for any $q, n\geq 0$. It implies that 
$$\pi _*\circ i _*\circ\rho^* _X :L^{hom} _{r-k}H
_{*+2(r-k)}(\Omega _X) _\mathbb{Q}\rightarrow L^{hom} _rH _{*+2r}(X) _\mathbb{Q}$$
is a surjective map.

Because $\Omega _X$ is a smooth variety of pure dimension $n-2k$ we know that $$L
_{n-2k}H _{2(n-2k)}(\Omega _X)\simeq H _{2(n-2k)}(\Omega _X)\simeq\mathbb{Z}^c$$
where c is the number of connected components of $\Omega _X$. We also know that the
cycle maps 
$$ cyc _{n-2k-1,q}:L _{n-2k-1}H _q(\Omega _X)\rightarrow H _q(\Omega _X)$$
are isomorphisms for $q\geq 2(n-2k)-1$ and monomorphisms for $q=2(n-2k-1)$ (\cite{F}). This
implies that $L^{hom} _{r-k}H _{*+2(r-k)}(\Omega _X) _\mathbb{Q}=0$ for $r\leq k$
and $r\geq n-k-1$. Because $\pi _*\circ i _*\circ\rho^* _X$ is a surjective map  we
conclude that $L^{hom} _rH _{*+2r}(X) _\mathbb{Q}=0$ for any $r\leq k$ or  $r\geq
n-k-1$.

We proved in Theorem \ref{the2} that $dj^*=0$ on $L^{hom} _*H _*(Z)$. We can see
that the constant $N=d(deg\pi)$ has the property that
$NL^{hom} _rH _{*+2r}(X)=0$ for any $r\leq k$ or any $r\geq n-k-1$.
\end{proof}
Theorem \ref{t3}, together with Proposition \ref{p1}, gives the following corollary:
\begin{corollary}
\label{coro12}
For any  smooth generic hypersurface $X\subset \mathbb{P}^{n+1}$ of degree $3\leq d\leq n+1$ with
$dim(\Omega _X(k))\geq n-2k$ we have 
$$L _rH _*(X) _\mathbb{Q}\simeq H _*(X^{an}) _\mathbb{Q}$$
for any $r\leq k$ and $r\geq n-k-1$ and $*\neq dim(X)=n$
\end{corollary}
The following theorem studies the surjectivity of the rational generalized cycle maps of $X$ into the middle dimension homology.
\begin{theorem}
\label{t4}
Let $X\subset \mathbb{P}^{n+1}$ be a smooth hypersurface of degree $3\leq d\leq n+1$ with $dim(\Omega
_X(k))\geq n-2k$. Then 
$$cyc _{r,n}\tensor \mathbb{Q}:L _rH _n(X) _\mathbb{Q}\simeq H _n(X^{an}) _\mathbb{Q}$$
for any $r\leq k$ or $r\geq n-k$. Moreover, we can fix a nonzero constant $N$ such that $NC _{r,n}(X)=0$ for such $r$.
\end{theorem}
\begin{proof}
 From Corollary \ref{c1}, we have that 
$$C _{r,*+2r}(X')\simeq \oplus _{l=0}^{k-1}C _{r-k+l+1,*+2r-2k+2l+2}(\Omega
_Z)\oplus C _{r-k,*+2r-2k}(\Omega _X)$$
As above, the Equality (\ref{e1}) gives that the induced map
$$C _{r-k,*+2r-2k}(\Omega _X) _{\mathbb{Q}}\stackrel{\pi _*\circ i _*\circ\rho _X^*}{\rightarrow} C _{r,*+2r}(X)
_{\mathbb{Q}}\rightarrow C _{r,*+2r}(X) _{\mathbb{Q}}/Im(j^*)\tensor \mathbb{Q}$$
is a surjective map. But Corollary \ref{c2} shows that $dj^*=0$ is the zero map on the cokernel of the cycle map $cyc _{q,n}$. This implies that the map
$$C _{r-k,*+2r-2k}(\Omega _X) _{\mathbb{Q}}\stackrel{\pi _*\circ i _*\circ\rho _X^*}{\rightarrow} C _{r,*+2r}(X) _{\mathbb{Q}}$$
is a surjective map. Let $*=n-2r\geq 0$. Then we get the surjection $C _{r-k,n-2k}(\Omega
_X) _{\mathbb{Q}}\rightarrow C _{r,n}(X) _{\mathbb{Q}}$. Because for $r\leq k$ or
$r\geq n-k$ we have $C _{r-k,n-2k}(\Omega _X) _\mathbb{Q}=0$, we also have $C
_{r,n}(X) _\mathbb{Q}=0$. As in Theorem (\ref{t3}), $N=d(deg\pi)$ has the property that $NC _{r,n}(X)=0$ for any $r\leq k$ or $r\geq n-k$.
\end{proof}
 
We sketch now another way to prove that the cycle maps $cyc _{k,n}:L _kH _n(X) _\mathbb{Q}\rightarrow
H _n(X^{an}) _\mathbb{Q}$ are surjective maps for any $X\subset \mathbb{P}^{n+1}$ generic hypersurface of degree $3\leq d\leq n+1$ and
with the property that $dim(\Omega _X(k))\geq n-2k$. We can notice below that the cylindrical homomorphism on homology appears in a very natural way in the context of Lawson homology. We will not insist on all the details of this remark as this alternative proof is not important for the results of this paper.

 We may identify $\Omega
_X(k)$ with $\mathcal{C} _{k,1}(X)$, the Chow variety of $k-$ dimensional subvarieties of
degree 1. We can see $P(X)\subset \mathcal{C} _{k,1}(X)\times X$ as an algebraic
cycle of dimension $n-k$. Its action gives a map 
$$ \phi _*:L _{r-k}H _{*+2(r-k)}(\Omega _X(k))\rightarrow L _rH _{*+2r}(X).$$
 This map can also be seen as being the action given by the Chow correspondence $id:
\mathcal{C} _{k,1}(X)\rightarrow  \mathcal{C} _k(X)$ (see \cite{Fr}). This means that
the cylindrical homomorphism $\phi _*$ on homology (or $r=0$ above) is given by the following composition:
\begin{equation}
\label{e2}
H _{n-2k}(C _{k,1}(X))\rightarrow \oplus _d H _{n-2k}(C _{k,d}(X))\rightarrow H
_{n-2k}(Z _k(X))\stackrel{r}{\rightarrow} \pi _{n-2k}(Z _k(X))\stackrel{cyc _{k,n}}{\rightarrow} H _n(X^{an}).
\end{equation}
The map $r$ is a section of the Hurewitz map of the topological group $Z _k(X)$.
Moreover 
$$\begin{CD}
      H _{n-2k}(\Omega _X) _\mathbb{Q} @>i _*  >> H _{n-2k}(\Omega _X(k))
_\mathbb{Q} \\
                     @V\pi _{X*}\circ \phi^* _X VV         @VV\phi _*V\\
      H _n(X^{an}) _\mathbb{Q}  @>=>> H _n(X^{an}) _\mathbb{Q}\\
\end{CD}$$
is a commutative diagram with $i _*$ surjective. We see that if $\pi _{X*}\circ
\phi^* _X$ is surjective then $\phi _*$ is surjective. From the composition
(\ref{e2}) we see that if $\phi _*$ is surjective then the cycle map
   $$cyc _{k,n}:L _kH _n(X) _ \mathbb{Q}=\pi _{n-2k}(Z _k(X)) _\mathbb{Q}\rightarrow H _n(X) _\mathbb{Q}$$ 
is a surjective map. But $\pi _{X*}\circ \phi^* _X$  is surjective for any generic
smooth hypersurface $X$ with the property that $dim(\Omega _X(k))\geq n-2k$ (\cite{Le}).

\begin{remark}
The above composition of maps (\ref{e2}) points toward a possible strategy to prove the Generalized Hodge conjecture for any generic hypersurface $X\subset \mathbb{P}^{n+1}$ of degree $d\leq n+1$, not only just for those fulfilling the condition  $dim(\Omega _X(k))\geq n-2k$. We see that it would be enough, for any generic hypersurface $X\subset \mathbb{P}^{n+1}$ of degree $d\leq n+1$, to prove the much weaker condition (weaker than the condition $\phi _*$ surjective map): the map
$$\oplus _d H _{n-2k}(C _{k,d}(X))\rightarrow H _n(X^{an})$$
is a surjective map.
\end{remark}
 We may apply the above results to study semi-topological invariants of generic  cubic hypersurfaces. 
\begin{corollary}
\label{qc}
Let $X$ be a smooth generic cubic fivefold. Then:

1)$L _1H _2(X)\stackrel{\mathbb{Q}}{\simeq} H _2(X^{an})\simeq\mathbb{Z}$.

2)$L _1H _3(X)\stackrel{\mathbb{Q}}{\simeq}H _3(X^{an})=0$.

3)$L _iH _4(X)\stackrel{\mathbb{Q}}{\simeq} H _4(X^{an})\simeq\mathbb{Z}$ \textnormal{for any}\hspace{1mm} $0\leq i\leq 2$.

4)$L _2H _5(X)\stackrel{\mathbb{Q}}{\simeq} L _1H _5(X){\simeq} H _5(X^{an})$.

5)$L _iH _6(X)\simeq H _6(X^{an})\simeq\mathbb{Z}$ \textnormal{for any}\hspace{1mm} $0\leq i\leq 3$.

6)$L _iH _7(X)\simeq H _7(X^{an})=0$ \textnormal{for any}\hspace{1mm} $0\leq i\leq 3$.

7)$L _iH _8(X)\simeq H _8(X^{an})\simeq\mathbb{Z}$ \textnormal{for any}\hspace{1mm} $0\leq i\leq 4$.

8)$L _iH _9(X)\simeq H _9(X^{an})=0$ \textnormal{for any}\hspace{1mm} $0\leq i\leq 4$.

9)$L _iH _{10}(X)\simeq H _{10}(X^{an})\simeq\mathbb{Z}$ \textnormal{for any}\hspace{1mm} $0\leq i\leq 5$.

10)$L _pH _n(X)=0$ for any $n>2dim(X)=10$.

In particular, any generic smooth cubic fivefold fulfills Suslin's conjecture. 
\end{corollary}
\begin{proof}
For a generic smooth cubic fivefold we have $k=[\frac{n+1}{d}]=2$. For this $k$, we may apply the same arguments as in the proof of the case of a generic cubic eightfold given below. 
\end{proof}

\begin{corollary} Let $X$ be a smooth generic cubic fivefold. Then $$K^{sst} _*(X) _\mathbb{Q}\simeq ku^{-*} _\mathbb{Q}(X^{an})$$ for any $*\geq 0$. 
\end{corollary}
\begin{proof}
The corollary follows directly from Corollary \ref{qc} and  Theorem \ref{ss}.
\end{proof}
\begin{corollary} 
\label{g8}
Let $X$ be a generic smooth cubic hypersurface of dimension 8. Then:

1)$L _1H _2(X)\stackrel{\mathbb{Q}}{\simeq} H _2(X^{an})\simeq\mathbb{Z}.$

2)$L _1H _3(X)\stackrel{\mathbb{Q}}{\simeq} H _3(X^{an})=0.$

3)$L _iH _4(X)\stackrel{\mathbb{Q}}{\simeq} H _4(X^{an})\simeq\mathbb{Z}$ \textnormal{for any}\hspace{1mm} $0\leq i\leq 2$.

4)$L _iH _5(X)\stackrel{\mathbb{Q}}{\simeq}H _5(X^{an})=0$ \textnormal{for any}\hspace{1mm} $0\leq i\leq 2$.

5)$L _iH _6(X)\stackrel{\mathbb{Q}}{\simeq}H _6(X^{an})\simeq\mathbb{Z}$ \textnormal{for any}\hspace{1mm} $0\leq i\leq 3$.

6)$L _3H _7(X)\stackrel{\mathbb{Q}}{\simeq}H _7(X^{an})=0$ \textnormal{for any}\hspace{1mm} $0\leq i\leq 3$.

7)$L _4H _8(X)\stackrel{\mathbb{Q}}{\hookrightarrow}L _iH _8(X)\stackrel{\mathbb{Q}}{\simeq} L _1H _8(X)\simeq H _8(X^{an})$ \textnormal{for any}\hspace{1mm} $0\leq i\leq 3$.

8)$L _4H _9(X)\stackrel{\mathbb{Q}}{\simeq}L _3H _9(X)\stackrel{\mathbb{Q}}{\simeq}L _2H _9(X)\simeq L _1H _9(X)\simeq H _9(X^{an})=0.$

9)$L _5H _{10}(X)\stackrel{\mathbb{Q}}{\simeq}L _4H _{10}(X)\stackrel{\mathbb{Q}}{\simeq} L _iH _{10}(X)\simeq H _{10}(X^{an})\simeq\mathbb{Z}$ \textnormal{for any}\hspace{1mm} $0\leq i\leq 3$.

10)$L _5H _{11}(X)\stackrel{\mathbb{Q}}{\simeq}L _iH _{11}(X)\simeq H _{11}(X^{an})$ \textnormal{for any}\hspace{1mm} $0\leq i\leq 4$.

11)$L _iH _{12}(X)\simeq H _{12}(X^{an})\simeq\mathbb{Z}$ \textnormal{for any}\hspace{1mm} $0\leq i\leq 6$.

12)$L _iH _{13}(X)\simeq H _{13}(X^{an})=0$ \textnormal{for any}\hspace{1mm} $0\leq i\leq 6$.

13)$L _iH _{14}(X)\simeq H _{14}(X^{an})\simeq\mathbb{Z}$ \textnormal{for any}\hspace{1mm} $0\leq i\leq 7$.

14)$L _iH _{15}(X)\simeq H _{15}(X^{an})=0$ \textnormal{for any}\hspace{1mm} $0\leq i\leq 7$.

15) $L _iH _{16}(X)\simeq H _{16}(X^{an})\simeq\mathbb{Z}$ \textnormal{for any}\hspace{1mm} $0\leq i\leq 8$.

16) $L _rH _n(X)=0$ for any $n>2\textnormal{dim}(X)=16$. 

In particular, any smooth generic cubic eightfold fulfills Suslin's conjecture.

\end{corollary}
\begin{proof}
 For a generic cubic eightfold we have $k=[\frac{8+1}{3}]=3$. Using Corollary \ref{coro12}, we can conclude that the cycle maps 
$$cyc _{q,n}\tensor \mathbb{Q}:L _qH _n(X) _\mathbb{Q}\rightarrow H _n(X^{an}) _\mathbb{Q}$$ 
with $n\neq 8$, are isomorphisms for any $0\leq q\leq 8$, $n\geq 2q$. Because $n\neq 8$ (=middle dimension), $cyc _{q,n}$ are surjective maps with integer coefficients (Corollary \ref{coro2}). Applying Theorem \ref{t3} to $k=3$, we conclude that the kernels of
$$cyc _{q,n}: L _qH _n(X)\rightarrow H _n(X^{an})$$
have finite torsion for any $n\geq 2q$.

 According to Proposition \ref{propo}, Ker($cyc _{q,n}$) is divisible for $n\geq 7+q$, so we can conclude that the cycle maps with integer coefficients 
$$cyc _{q,n}:L _qH _n(X) \rightarrow H _n(X^{an})$$ 
are isomorphisms for any $n\geq 7+q$, $n\neq 8$ and 
$$cyc _{1,8}:L _1H _8(X)\hookrightarrow H _8(X^{an})$$
is an injective map with integer coefficients.
Looking to the rational cycle maps
$$cyc _{r,8}:L _rH _8(X) _\mathbb{Q}\rightarrow H _8(X^{an}) _\mathbb{Q}$$
we notice that they are isomorphisms for any $r\leq 3$ (Theorem \ref{t4}) and monomorphism for $r=4$ (Theorem \ref{t3}). Moreover, according to Proposition \ref{propo} and Theorem \ref{t4}, the cycle map
$$cyc _{1,8}:L _1H _8(X)\rightarrow H _8(X^{an})$$
is surjective. Thus $L _1H _8(X)\simeq H _8(X^{an})$.
 To conclude the theorem, we only need to prove that $$L _6H _{12}(X)\simeq L _5H _{12}(X).$$
This map is injective because on codimension 2 cycles the algebraic equivalence coincides with the homological equivalence for smooth projective rationally connected varieties (\cite{BS}, Theorem 1). It is also surjective, according to Corollary \ref{coro2}, because
$$cyc _{6,12}:L _6H _{12}(X)\stackrel{s}{\rightarrow} L _5H _{12}(X)\simeq H _{12}(X^{an})$$ 
is a surjective map.
\end{proof}
 The case of a generic smooth cubic sixfold can be treated similarly.
\begin{corollary} 
\label{c4}
Let $X$ be a generic cubic sixfold or eightfold. Then 
$$K^{sst} _*(X) _\mathbb{Q}\simeq ku^{-*} _\mathbb{Q}(X^{an})$$
for any $*\geq 1$ and
$$K^{sst} _0(X) _\mathbb{Q}\hookrightarrow ku^0 _\mathbb{Q}(X^{an}).$$
\end{corollary}
\begin{proof} Let $X$ be a generic cubic hypersurface. 
Then the  corollary follows from Corollary \ref{g8} and the following  theorem:
\begin{theorem} (\cite{Voin}, \cite{FHW})

Let $X$ be a smooth quasi-projective complex variety of dimension $d$. Let $A$ be an abelian group and $r\leq 0$. Then if 
$$cyc^{q,n}:L^qH^n(X,A)\rightarrow H^n(X^{an},A)$$
is an isomorphism for $n-2q\leq r-1$ and a monomorphism for $n-2q\leq r$, then the map $$K^{sst} _i(X,A)\rightarrow ku^{-i}(X^{an},A)$$
is an isomorphism for $i\geq -r+1$ and a monomorphism for $i=-r$.
\end{theorem}
The case of a cubic sixfold can be treated similarly.

\end{proof}

Recall that for a smooth projective variety $X$ we have the following maps, starting from the cycle spaces and ending in the singular homology groups of $X$:
$$Z _r(X)\stackrel{\pi}{\rightarrow} \ L _rH _{2r}(X)=\pi _0(Z _r(X))\stackrel{s}{\rightarrow} L _{r-1}H _{2r}(X)\stackrel{s}{\rightarrow}...\stackrel{s}{\rightarrow}H _{2r}(X^{an}).$$
If  $S _iZ _r(X)=Ker(s^i\circ \pi)$ then
$$0\subset S _1Z _r(X)/S _0Z _r(X)\subset ..\subset \textnormal{Griff} _r(X)=S _rZ _r(X)/S _0Z _r(X)$$
gives the s-filtration (\ref{s-filtration}) recalled in the first section of the paper. 

We use the above results to show examples of varieties for which the lowest step in the s-filtration of  Griff$_r(X)\tensor\mathbb{Q}$ is an infinitely generated rational vector space.    
\begin{corollary}
\label{c46}
Let $X$ be a generic cubic sevenfold. Then we have
$$L _3H _6(X) _\mathbb{Q}\stackrel{s}{\rightarrow} L _2H _6(X) _\mathbb{Q}\stackrel{s}{\simeq} L _1H _6(X)
_\mathbb{Q}\stackrel{s}{\simeq} H _6(X^{an}) _\mathbb{Q}\simeq \mathbb{Q}$$
where the only non-isomorphism s-map from the sequence has infinitely generated
kernel. Consider $Y$ to be a sixfold of  large degree obtained from intersecting $X$ with a generic hypersurface of 
large degree. Then 
$$Ker(L _2H _4(Y) _\mathbb{Q}\stackrel{s}{\rightarrow} L _1H _4(Y) _\mathbb{Q})$$
is infinitely generated. Moreover, the Abel-Jacobi map restricted to this kernel is zero. 

Notice that this kernel is the lowest step in the s-fitration of \textnormal{Griff}$_2(Y)\tensor \mathbb{Q}$.
\end{corollary}
\begin{proof}
For a generic cubic sevenfold we have $k=[\frac{7+1}{3}]=2$. Using Corollary \ref{coro12}
we obtain that 
$$L _2H _6(X) _\mathbb{Q}\simeq L _1H _6(X) _\mathbb{Q}\simeq H _6(X^{an}) _\mathbb{Q}.$$ 
This shows that 
\begin{align*}
\textnormal{Griff} _3(X) _\mathbb{Q}&=Ker(L _3H _6(X) _\mathbb{Q}\stackrel{cyc _{3,6}}{\rightarrow} H _6(X^{an}) _\mathbb{Q})\\
&=Ker(L _3H _6(X) _\mathbb{Q}\stackrel{s}{\rightarrow} L _2H _6(X) _\mathbb{Q}).
\end{align*}
Let 
$$\Phi^4 _X: \textnormal{Griff} _3(X)\rightarrow J^4(X)$$
be the Abel-Jacobi map on dimension three algebraic cycles of $X$ with values in the intermediate Jacobian $J^4(X)$ (\cite{Vo}).
Albano and Collino (\cite{AC}, Theorem 1) proved that Griff$_3(X) _\mathbb{Q}$ is infinitely
generated in the case of a cubic sevenfold showing that the image of this Abel-Jacobi
map $\Phi^4 _X($Griff$_3(X) _\mathbb{Q})$ is infinitely generated. This result implies the
first part of our corollary. 

Let $Y$ be the intersection of $X$ with a hypersurface of high degree in
$\mathbb{P}^8$. In (\cite{AC}, Theorem 2), it is proved that Griff$_2(Y) _\mathbb{Q}$ is
infinitely generated, but the Abel-Jacobi map on this Griffiths group is zero.
We show that 
  $$Ker(L _2H _4(Y) _\mathbb{Q}\stackrel{s}{\rightarrow} L _1H _4(Y) _\mathbb{Q})\subset \textnormal{Griff}
_2(Y) _\mathbb{Q}$$
is infinitely generated. Let 
$$\xi\in Ker _X(s\circ \pi: CH _3(X) _\mathbb{Q}\stackrel{\pi}{\rightarrow} L _3H _6(X)
_\mathbb{Q}\stackrel{s}{\rightarrow} L _2H _6(X)_\mathbb{Q} )=CH _3(X) _{hom}\tensor\mathbb{Q}$$ 
and let $R$ denote the map given by the following composition:
$$R:Ker _X(s\circ \pi)\rightarrow Ker _Y(s\circ \pi)\rightarrow Ker _Y(s\circ
\pi)/Ker _Y\pi=Ker(L _2H _4(Y)  _\mathbb{Q}\stackrel{s}{\rightarrow} L _1H _4(Y) _\mathbb{Q} ).$$
 Nori \cite{NO} defined the following groups for a smooth projective variety $X$:
 \begin{align*}
  A _jCH _r(X)=&\textnormal{subgroup in} \hspace{1mm} Z _r(X) \hspace{1mm} \textnormal{generated by those algebraic cycles that are rationally}\\ 
  &\textnormal{equivalent to cycles of the form}  \hspace{1mm} pr _{X*}(pr^* _Y(W).Z), \textnormal{where Y is a }\\
  & \textnormal{smooth projective variety,} \hspace{1mm} W\in S _jZ _j(Y) \hspace{1mm} \textnormal{and}  \hspace{1mm} Z\in Z _{r+dim(Y)-j}(X\times Y).
 \end{align*} We have  the following theorem of M.Nori \cite{NO}.
\begin{theorem}(Nori \cite{NO})
\label{Nori}
Let $X\subset \mathbb{P}^m$ be a smooth, projective variety and let $i:Y\hookrightarrow X$ be
the intersection of $X$ with $h$ general hypersurfaces of sufficiently large
degrees. Let $\xi\in CH^d(X)$ and let $\mu=i^*(\xi)$. Assume that $r+d<dim(Y)$. If
$\mu\in A _rCH _{dim(Y)-d}(Y) _ \mathbb{Q}$, then

(1) The cohomology class of $\xi$ vanishes in $H^{2d}(X,\mathbb{Q})$, and

(2) the Abel-Jacobi image of a non-zero multiple of $\xi$ belongs to $J^d _r(X)$

\end{theorem}

We use this theorem in the case $h=1$, $d=4$ and $r=0$. In this case  
$$A _0CH _2(Y)=CH^4 _{alg}(Y)=\{\textrm{algebraic cycles on Y of dimension 2 algebraically equivalent to zero}\}.$$ 
If $R(\xi)=\mu =0$ then $\xi\in J ^4 _0(X) _\mathbb{Q}$ by the Theorem \ref{Nori}. $J^4 _0(X)$ is by definition the ``largest" abelian subvariety of the intermediate Jacobian $J^4(X)$. For a generic cubic sevenfold $X$, we have $J ^4 _0(X)=0$ \cite{Vo}. This shows that 
$$Ker(R)\subset Ker(\Phi^4 _X: CH^4 _{hom}(X) _\mathbb{Q}\rightarrow J^4(X) _\mathbb{Q}).$$
But
\begin{align*}
Im(\Phi^4 _X)\simeq CH^4 _{hom}(X)/Ker(\Phi^4 _X)&\subset CH^4_{hom}(X)/Ker(R)\simeq Im(R)\subset \\
&\subset Ker(L _2H _4(Y)  _\mathbb{Q}\stackrel{s}{\rightarrow} L _1H _4(Y)  _\mathbb{Q}).
\end{align*}
 The image  $\Phi^4 _X(Ker _X(s\circ \pi) _\mathbb{Q})$ is infinitely generated because $\Phi^4 _X(\textnormal{Griff} _3(X)
_\mathbb{Q})$ is infinitely generated \cite{AC}. We obtain that 
$$ Ker(L _2H _4(Y)\tensor\mathbb{Q}\stackrel{s}{\rightarrow} L _1H _4(Y)\tensor\mathbb{Q})$$
is infinitely generated.
\end{proof}
\begin{remark}
The above corollary supports the conjecture that for high degree complete intersection varieties all s-maps, except those in Suslin's conjecture range of indices, have infinitely generated kernel.
\end{remark}
\begin{remark}
\label{eiel}
 Homotopy groups of topological spaces of cycles of dimension 3 on cubic sevenfolds are interesting particular cases in which Suslin's conjecture and the Friedlander-Mazur conjecture are not yet checked. For a generic cubic sevenfold $X$, Suslin's conjecture predicts that $L _3H _r(X)=\pi _{r-6}Z _3(X)$ is a finitely generated group for any $r\geq 9$. 
 
 Our methods check the Friedlander-Mazur conjecture for any space of algebraic cycles of a  generic cubic sevenfold $X$, except for the space of algebraic cycles on $X$ of dimension 3. That is, for a generic cubic sevenfold $X$, $L _rH _n(X)=0$ for any $r\neq 3$ and $n>2dim(X)=14$.
 
 Similar arguments show that homotopy groups of topological spaces of cycles of dimension 5 on generic cubic elevenfolds are interesting particular cases in which these two conjectures are not yet checked. For example,  for a generic cubic elevenfold, our methods show that $L _rH _n(X)=0$ for any $r\neq 5$ and $n>2dim(X)=22$.
    
\end{remark}
 In the end we make the following remark.
\begin{corollary}
\label{aju}
 For any $r\geq 2$, there is a smooth projective variety $X$ such that $S _1Z _r(X)\tensor \mathbb{Q}/S _0Z _r(X)\tensor \mathbb{Q}$ is an infinitely generated $\mathbb{Q}$-vector space and the restriction of the Abel-Jacobi map on this step of the s-filtration of \textnormal{Griff}$_r(X)\tensor\mathbb{Q}$ is zero.
\end{corollary}
\begin{proof}
 If $r=2$, we can take $X$ to be the above constructed generic complete intersection of dimension $6$. Let $r>2$ and let $X$ be the generic complete intersection of dimension 6 defined above and $P$ a smooth projective variety of dimension $m\geq 1$. Let $\{\gamma _i\in S _1Z _2(X)\tensor\mathbb{Q}/S _0Z _2(X)\tensor \mathbb{Q}\}$ be a set of linearly independent algebraic cycles. Suppose $$\{\gamma _i\times P\in S _1Z _{2+m}(X\times P)\tensor\mathbb{Q}/S _0Z _{2+m}(X\times P)\tensor \mathbb{Q}\}$$ is a set of linearly dependent algebraic cycles. This implies that there are some $a _i\in\mathbb{Q}$ such that $$\sum _ia _i(\gamma _i\times P)=0\in S _1Z _{2+m}(X\times P)\tensor\mathbb{Q}/S _0Z _{2+m}(X\times P)\tensor \mathbb{Q}.$$ According to (\cite{Fr}, Proposition 3.5) we can conclude that $$\sum _i a_i\gamma _i=0\in S _1Z _{2+m}(X)\tensor\mathbb{Q}/S _0Z _{2+m}(X)\tensor \mathbb{Q}$$ which is a contradiction with our choice of $X$ and $\gamma _i$. It implies that $S _1Z _{2+m}(X\times P)\tensor\mathbb{Q}/S _0Z _{2+m}(X\times P)\tensor \mathbb{Q}$ is an infinitely generated $\mathbb{Q}-$vector space. It is obvious that the Abel-Jacobi map applied to the cycles $\gamma _i\times P$ is zero, because it is zero on the cycles $\gamma _i$. This implies the corollary for $r>2$.
  
\end{proof}
\section{Appendix} In this appendix we will prove two technical propositions needed in the paper.  We will use the same notations introduced in the beginning of the second section. 


  

We prove the following proposition (in $\mathcal{H}^{-1}AbTop$):
\begin{proposition}
\label{Pu}
The map 
\begin{equation}
\label{P}
(\rho \times 1) _*\circ pr _1^*\mu^{k-1-l}\circ (\rho \times 1)^*: Z _r(\Omega _Z\times \mathbb{P}^m)\rightarrow Z _r(\Omega _Z \times \mathbb{P}^m)
\end{equation}
is the identity map if $l=0$ and the zero map if $l>0$. 

Here $pr _1: X'\times \mathbb{P}^m\rightarrow X'$ is the canonical projection on $X'$ and $m\geq -1$.
\end{proposition}
\begin{proof}
According to the projection formula for Gysin maps \cite{P} we have   
$$(\rho \times 1) _*\circ  pr _1^*\mu^{k-1-l}\circ (\rho \times 1)^*(\alpha)=[(\rho \times 1) _*pr _1^*\mu^{k-1-l}].\alpha$$
for any $\alpha\in Z _r(\Omega _Z\times \mathbb{P}^m)$ and $r\geq 0$. 

This means that the map (\ref{P}) is the action of the cycle $$\beta:=(\rho \times 1) _*pr _1^*\mu ^{k-1-l}$$ on $Z _r(\Omega _Z\times \mathbb{P}^m)$. Recall that $dim(X')=n$ and $dim(\Omega _Z)=n-k+1$. We have
$$pr _1^*\mu^{k-1-l} \in Z _{n-k+1+l+m}(X'\times\mathbb{P}^m)$$ 
so  $\beta \in Z _{n-k+1+l+m}(\Omega _Z \times \mathbb{P}^m)$. This implies that $\beta=0$ if $l>0$ and $\beta=a[\Omega _Z\times \mathbb{P}^m]$, for some integer $a$, in case  $l=0$. Thus, we conclude that  the map (\ref{P}) is zero if $l>0$ and  $a(id)$ if $l=0$. In particular, the map (\ref{P}) in case $l=0$ induces the multiplication by $a$ on all homotopy groups of $Z _r(\Omega _Z\times \mathbb{P}^m)$, for any choice of $r\geq 0$. 

In particular, for $r=m+n-k+1$, we have that 
 $$(\rho\times1) _*pr _1^*\mu^{k-1}(\rho\times 1)^*=a(id): A _r(\Omega _Z\times \mathbb{P}^m)\rightarrow A _r(\Omega _Z\times \mathbb{P}^m)$$
where $A _r(\Omega _Z\times \mathbb{P}^m)=\pi _0(Z _r(\Omega _Z\times \mathbb{P}^m))$ is the Chow group of algebraic cycles modulo algebraic equivalence. To avoid any possible confusions, we will write below $pr _{X'}$ for the above map $pr _1$.
Looking to the following cartesian diagram
$$ \begin{CD}
     X'\times \mathbb{P}^m @>pr _{X'}>> X'\\
           @V\rho\times 1 VV            @VV\rho V\\
    \Omega _Z\times \mathbb{P}^m @>pr _{\Omega _Z}>> \Omega _Z\\
\end{CD}$$
we can conclude the following equalities on Chow groups:
\begin{align*}
(\rho\times 1) _*\circ pr _{X'}^*\mu ^{k-1}\circ (\rho\times 1)^*\circ pr _{\Omega _Z}^*()&=(\rho\times 1) _*[pr _{X'}^*\mu^{k-1}.(pr _{\Omega _Z}\circ (\rho\times 1))^*()]\\
&=(\rho\times 1) _*\circ pr _{X'}^*[\mu^{k-1}.\rho^*()]\\
&\stackrel{(3)}{=}pr^* _{\Omega _Z} \circ\rho _*\circ [\mu^{k-1}.\rho^*()]=pr^* _{\Omega _Z}\circ \rho _*\circ \mu^{k-1}\circ \rho^*().
\end{align*}
 The equality (3) comes from the flat base-change formula (\cite{Ful}). Thus, we have the following commutative diagram
$$ \begin{CD}
     CH _{n-k+1}(\Omega _Z) @>pr _{\Omega _Z}^*>> CH _{m+n-k+1}(\Omega _Z\times \mathbb{P}^m) \\
           @V\rho _*\circ\mu^{k-1}\circ\rho^* VV            @VV(\rho \times 1) _*\circ pr _{X'}^*\mu^{k-1}\circ(\rho\times 1)^* V\\
    CH _{n-k+1}(\Omega _Z)@>pr _{\Omega _Z}^*>> CH _{m+n-k+1}(\Omega _Z\times \mathbb{P}^m) \\
\end{CD}$$

According to (\cite{Le}, Page 216) the left vertical arrow is the identity map. Because the right vertical arrow is $a(id)$ we have the following equality 
$$pr _{\Omega _Z}^*([\Omega _Z])=a pr _{\Omega _Z}^*([\Omega _Z]).$$
From the injectivity of $pr _{\Omega _Z}^*$ (\cite{Ful}, Corollary 3.1), we conclude that $a=1$.
\end{proof}
We have the following diagram of abelian groups with exact rows and surjective vertical maps (see diagram (\ref{dia2})):
\begin{equation*}
\xymatrix@-1.3pc{
 \ar[r]&\pi _{*+1}(Z _r(P(X))) \ar[r]^{i _*}\ar[d]^{\alpha}   &   \pi _{*+1}(Z _r(X')) \ar[r]\ar[d]^{(1)} _{\mp f(T)\circ \lambda _1}   &  \pi _{*+1}(Z _r(X' {\setminus P(X)}))  \ar[r]\ar[d]^{(2)} _{\beta\simeq}& \\ 
\ar[r]&   \oplus _{l=0}^{k-1} \pi _{*+1}Z _{r-k+1+l}(\Omega _X) \ar[r]  &  \oplus _{l=0}^{k-1} \pi _{*+1}Z _{r-k+1+l}(\Omega _Z)  \ar[r]  & \oplus _{l=0}^{k-1} \pi _{*+1}Z _{r-k+1+l}(\Omega _Z\setminus\Omega _X)\ar[r]& \\
}
\end{equation*}

\begin{equation*}
\xymatrix@-1.2pc{
 \ar[r]& \pi _*(Z _r(P(X)))\ar[r]^{i _*}\ar[d] _{\alpha} ^{(3)} & \pi _*(Z _r(X'))\ar[r]\ar[d] _{\mp f(T)\circ \lambda _1} ^{(4)} &  \pi _*(Z _r(X'\setminus P(X)))\ar[r] \ar[d] _{\beta} ^{\simeq} &\\
\ar[r]& \oplus _{l=0}^{k-1} \pi _*Z _{r-k+1+l}(\Omega _X) \ar[r]  &  \oplus _{l=0}^{k-1} \pi _*Z _{r-k+1+l}(\Omega _Z)  \ar[r]  & \oplus _{l=0}^{k-1} \pi _*Z _{r-k+1+l}(\Omega _Z\setminus\Omega _X)\ar[r]& \\
}
\end{equation*}

Notice that, according to Theorem \ref{projb}, for the $\mathbb{P}^k-$bundle $\rho _X:P(X)\rightarrow \Omega _X$ we have  
\begin{align*}
\pi _* Z _r(P(X))=\sum _{j=0}^{k}h^{k-j}.(\rho _X)^*(\pi _*(Z _{r-j}(\Omega _X)))&\stackrel{\phi ^{-1}}{\simeq}  \oplus _{j=0}^{k-1}\pi _*Z _{r-j}(\Omega _X)\oplus \pi _*Z _{r-k}(\Omega _X)\\
&\stackrel{(1)}{=}  \oplus _{l=0}^{k-1}\pi _*Z _{r-k+1+l}(\Omega _X)\oplus \pi _*Z _{r-k}(\Omega _X).
\end{align*}
The Equality (1) is from rewriting  the indices (making $l=k-1-j$).  Thus, for $\alpha:   \pi _*(Z _r(P(X)))\rightarrow \oplus _{l=0}^{k-1} \pi _*Z _{r-k+1+l}(\Omega _X)$ we have
\begin{equation}
\label{eqK}
 Ker(\alpha)=\rho _X^*(\pi _*Z _{r-k}(\Omega _X)).
\end{equation}
The following proposition is part of the proof of Theorem \ref{the1}.
\begin{proposition} 
\label{dchas}
The map 
$$\lambda _2+i _*\circ \rho^* _X:\oplus _{l=0}^{k-1} \pi _*Z _{r-k+1+l}(\Omega _Z)\oplus \pi _*Z _{r-k}(\Omega _X)\simeq\pi _*(Z _r(X'))$$
is an isomorphism.
\end{proposition}
\begin{proof} The proof is identical with the proof of Proposition 3.2 in \cite{Le}, with the only change that instead of singular cohomology, we use homotopy of cycle spaces. For the sake of completeness, we will sketch below the details. We will prove the following four statements:

$1) \pi _*Z _r(X')=\lambda _2(\oplus _{l=0}^{k-1}\pi _*Z _{r-k+1+l}(\Omega _Z))+i _*(\pi _*Z _r(P (X)))$.

$2) i _*:\rho^* _X(\pi _*Z _{r-k}(\Omega _X))\hookrightarrow \pi _*Z _r(X')$ is injective.

$3) \lambda _2(\oplus _{l=0}^{k-1}\pi _*Z _{r-k+1+l}(\Omega _Z)) \cap i _*(\rho ^*_X\pi _*Z _{r-k}(\Omega _X))=0.$

$4) \pi _*Z _r(X')\simeq \lambda _2(\oplus _{l=0}^{k-1}\pi _*Z _{r-k+1+l}(\Omega _Z))\oplus i _*\rho^* _X(\pi _*Z _{r-k}(\Omega _X)).$

The conclusion of our proposition will  follow directly from 4) and 2).

{\textbf{Proof of 1)}}

1) follows from the split short exact sequence (with $\lambda _2$ section map):
$$0\rightarrow Ker(\mp f(T)\lambda _1)\rightarrow \pi _*Z _r(X')\rightarrow \oplus _{l=0}^{k-1}\pi _*Z _{r-k+1+l}(\Omega _Z)\rightarrow 0$$
and from the fact that $Ker(\mp f(T)\lambda _1)\subset Im(i _*)$.

{\textbf{Proof of 2)}} 

The statement follows from a simple diagram chase in diagram (1) and (2), using the Equality (\ref{eqK}) and that $\mp f(T)\lambda _1:\pi _{*+1}(Z _r(X'))\rightarrow \oplus _{l=0}^{k-1} \pi _{*+1}(Z _{r-k+1+l}(\Omega _Z)) $ is a surjective map.

{\textbf{Proof of 3)}}

Let $y\in \pi _*Z _{r-k}(\Omega _Z)$ and suppose there is $z\in \oplus _{l=0}^{k-1}\pi _*Z _{r-k+1+l}(\Omega _Z)$ such that $\lambda _2(z)=i _*\rho ^* _X(y)$. Applying $\rm f(T)\lambda _1$ on this equality gives $z=0$.

{\textbf{Proof of 4)}}

The proof of 4) is word for word the proof of Lemma 3.6 from \cite{Le} so we skip this diagram chase.
 \end{proof}



\bibliographystyle{plain}
\bibliography{mircea}

\end{document}